\documentclass[11pt,twoside]{amsart}
\usepackage[all]{xy}
        \usepackage {amssymb,latexsym,amsthm,amsmath,mathtools,dsfont}
        \usepackage{enumitem,color}

\usepackage{mathtools}

\long\def\symbolfootnote[#1]#2{\begingroup%
\def\thefootnote{\fnsymbol{footnote}}\footnote[#1]{#2}\endgroup}

\newcommand{\m}{\textup{M}}

\newcommand{\f}{\mathbb F}
\newcommand{\C}{\mathcal C}

\newcommand{\GL}{\textup{GL}}
\newcommand{\PGL}{\textup{PGL}}
\newcommand{\p}{\textbf{P}}
\newcommand{\SL}{\textup{SL}}

\makeatletter
\def\imod#1{\allowbreak\mkern10mu({\operator@font mod}\,\,#1)}
\makeatother
\makeatletter
\renewcommand*\env@matrix[1][*\c@MaxMatrixCols c]{%
  \hskip -\arraycolsep
  \let\@ifnextchar\new@ifnextchar
  \array{#1}}
\makeatother

\newtheorem{theorem}{Theorem}[section]
\newtheorem{lemma}[theorem]{Lemma}
\newtheorem{corollary}[theorem]{Corollary}
\newtheorem{proposition}[theorem]{Proposition}
\newtheorem*{theorem*}{Theorem}
\theoremstyle{definition}
\newtheorem{definition}[theorem]{Definition}
\newtheorem{remark}[theorem]{Remark}

\numberwithin{equation}{section}

\newcommand{\ignore}[1]{}

\newcommand{\mynote}[1]{}
\begin{document}
\setcounter{section}{0}
\title{Unit group of $\f_{p^k}\SL(3,2),p\geq 11$}
\author{Namrata Arvind, Saikat Panja}
\email{namchey@gmail.com, panjasaikat300@gmail.com}
\address{IISER Pune, Dr. Homi Bhabha Road, Pashan, Pune 411 008, India}
\thanks{$^*$The first named author is partially supported by the IISER Pune research fellowship where as the second author has been partially supported by supported by NBHM fellowship.}


\begin{abstract}
We provide the structure of the unit group of $\f_{p^k}(\SL(3,2))$, where $p\geq 11$ is a prime and $\SL(3,2)$ denotes the $3\times 3$ invertible matrices over $\f_2$.
\end{abstract}
\maketitle
\section{Introduction}
Let $q=p^k$ for some prime $p$ and $k\in\mathbb{N}$. Let $\f_q$ denote the finite field of cardinality $q$. For any group $G$, let $\f_qG$ denotes the group algebra of $G$ over $\f_q$. For basic notations and results on the subject of study, we refer the readers to the classic by Milies and Sehgal \cite{MS}. The group of units of $\f_qG$ has many applications. As an application of the unit groups of matrix rings, Hurley has proposed the constructions of convolutional codes (See \cite{H1},\cite{H2},\cite{HH1},\cite{HH2}). The structure of unit group
can also be used to deal with some problems in combinatorial number theory as well (See \cite{GGK}). This has encouraged a lot of researchers to find out the explicit structure of the group of units of $\f_qG$.
 
 A substantial amount of work has been done to find the structure of the algebra $\f_qG$, and also of the group of units of these algebras. For example in \cite{S}, the author has described units of $\f_qG$, where $G$ is a $p$-group. In a recent paper \cite{BLP} the authors have discussed the groups of units for the group algebras over abelian groups of order $17$ to $20$.
 Howerver the complexity of the problem increases with increase in the size of the group and the number of conjugacy classes it has. For more, one can check \cite{MSS2}, \cite{MA},\cite{TG} et cetera. 
 
 Very little is known for $\f_q G$, when $G$ is a non-Abelian simple group. For the case $G=A_5$, this has been discussed in \cite{MSS1}. The next group in the family of non-Abelian simple groups is the group $\SL(3,2)$. In Theorem \ref{s009} of this article we give a complete description of the unit group of $\f_q\SL(3,2)$ for $p\geq 11$.

Rest of the article is organized as follows. In section $2$, we give the known results which we will be using in subsequent sections. In section $3$ we discuss about some simple components of the Artin-Wedderburn decomposition of the group algebra. Next in  section $4$, we deduce the main result. We discuss some observations and conclude the paper by mentioning some remarks, in section $5$.


\section{Preliminaries}
First we fix some notations. We adopt already mentioned notations from section $1$. For an extension field $E/\f_q$, $\textup{Gal}(E/\f_q)$ denotes the Galois group of the extension. For $n\in\mathbb{N}$ the notation
$\m(n,R)$ denotes the full matrix ring of $n\times n$ matrices over $R$ where as $\GL(n,R)$ will denote the set of all invertible matrices in $\m(n,R)$. For a ring $R$, the set of units of $R$ will be denoted as $R^\times$. The center of a ring $R$ will be denoted as $Z(R)$. If $G$ is a group and $g\in G$, then $[g]$ will denote the conjugacy class of $g$ in $G$. For the group ring $\f_q G$, the group of units will be denoted as $\mathcal{U}(\f_q G)$. For the notations on projective spaces, we follow \cite{H}.

We say an element $g$ $\in$ $G$ is a $p'$-element if the order of $g$ is not divisible by $p$. 
Let $e$ be the exponent of the group $G$ and $\eta$ be a primitive $r$th root of unity, where $e=p^fr$ and $p\nmid r$. Let \begin{equation*}
    I_{\f_q} = \left\{l~(\text{mod }e):\text{  there exists } \sigma \in \textup{Gal}(\f_q(\eta)/\f_q)\text{ satisfying } \sigma(\eta) = \eta^l\right\}.
\end{equation*}
\begin{definition}
    For a $p'$-element $g \in G$, the cyclotomic $\f_q$-class of $g$, denoted by $S_{\f_q}(\gamma_g)$ is defined as $\left\{\gamma_{g^l} : l \in I_{\f_q} \right\}$ where $\gamma_{g^l}\in\f_q G$ is the sum of all conjugates of $g^l$ in $G$.  
\end{definition}
Then we have the following results which are crucial in determining the Artin-Wedderburn decomposition of $\f_qG$.
\begin{lemma}{\cite[Proposition 1.2]{F}}\label{n001}
The number of simple components of $\f_qG/J(\f_qG)$ is equal to the number of cyclotomic $\f_q$-classes in $G$.
\end{lemma}
\begin{lemma} {\cite[Theorem 1.3]{F}}\label{n002}
Let $n$ be the number of cyclotomic $\f_q$-classes in $G$. If $L_1,L_2,\cdots,L_n$ are the simple components of $Z(\f_qG/J(\f_qG))$ and $S_1,S_2,\cdots,S_n$ are the cyclotomic $\f_q$-classes of $G$, then with a suitable reordering of the indices,
\begin{equation*}
    |S_i| = [L_i:\f_q].
\end{equation*}
\end{lemma}
\begin{lemma}\cite[Lemma 2.5]{MS2}\label{n003}
Let $K$ be a field of charecteristic $p$ and let $A_1$, $A_2$ be two finite dimensional $K$-algebras. Assume $A_1$ to be semisimple. If $h$ : $A_2$ $\longrightarrow$ $A_1$ is a surjective homomorphism of $K$-algebras, then there exists a semisimple $K$-algebra $l$ such that $A_2/J(A_2)$ $=$ $l\oplus A_1$.  
\end{lemma}
We will be using various descriptions of $\SL(3,2)$ in the sequel, which are well known. From \cite{CCNPW}, it is known that 
\begin{align*}
    \SL(3,2)=\GL(3,2)\cong \PGL(2,7)\cong\textup{PSL}(2,7).
\end{align*}
We have an embedding of $\SL(3,2)$ inside $S_8$ as follows:
\begin{align*}
    \SL(3,2)\cong\langle(3,7,5)(4,8,6),(1,2,6)(3,4,8)\rangle.
\end{align*}
This group has $7$ conjugacy classes and using \cite{GAP2021}, we have the following table:
\begin{align*}
    \begin{tabular}{|c|c|c|c|}
    \hline
    \text{Class}&\text{Representative}&\text{Order}&\text{No. of elements}\\\hline
       $\C_1$  & $\alpha_1=$ $(1)$& $1$ &$1$\\\hline
        $\C_2$ & $\alpha_2=$ $(1,2)(3 ,4)(5, 8)(6, 7)$& $2$ &$21$\\\hline
        $\C_3$ & $\alpha_3=$ $(3, 5, 7)(4, 6, 8)$ & $3$ &$56$\\\hline
        $\C_4$ & $\alpha_4=$ $(1, 2, 3, 5)(4, 8, 7, 6)$ & $4$ &$42$\\\hline
        $\C_5$ & $\alpha_5=$ $(2 ,3, 5, 4, 7, 8, 6)$ & $7$ &$24$\\\hline
        $\C_6$ & $\alpha_6=$ $(2 ,4, 6, 5, 8, 3, 7)$ & $7$ &$24$\\\hline
    \end{tabular}.
\end{align*}
We note down the following relations
\begin{equation}\label{ne010}
    [\alpha_5] = [\alpha_5^2] = [\alpha_5^4] .
\end{equation} 
and \begin{equation}\label{ne011}
    [\alpha_6] = [\alpha_5^3] = [\alpha_5^5] = [\alpha_5^6] = [\alpha_6].
\end{equation}
\section{On some simple components of $\f_qG$}
The next few lemmas are crucial for determining the different $n_i$'s occurring in the Artin-Wedderburn decomposition of $\f_q\SL(3,2)$.
\begin{lemma}\label{s001}
Let $G$ be a group of order $n$ and $\f$ be a field of characteristic $p>0$. Let $G$ acts on a finite set $X=\{1,2,\cdots,k\}$ doubly transitively. Set $G_{i}=\{g\in G:g\cdot i=i\}$ and $G_{i,j}=\{g\in G:g\cdot i=i,g\cdot j=j\}$. Then the $\f G$ module
\begin{align*}
    W=\left\{x\in\f^k:\displaystyle\sum\limits_{i=1}^kx_i=0,i\in X\right\}
\end{align*} 
is an irreducible $\f G$ module if $p\nmid k,p\nmid |G_{1,2}|$.
\end{lemma} 
\begin{proof}
Let $U\subseteq W$ be a non-zero invariant space under the action of $G$. Since the action is doubly transitive, it is enough to show that we have $(1,-1,\underbrace{0,\ldots,0}_{(k-2)\text{ times}})\in U$.

Let $x=(x_1,x_2,\ldots,x_n)\in U$ be nonzero. Then we can assume that $x_1\neq 0$, since $G$ acts transitively on $X$. Considering the element
$y=\displaystyle\sum\limits_{g\in G_1}gx\in U$,
we see that 
\begin{align*}
y_1&=|G_1|x_1\\
y_2&=y_3=\cdots=y_n\\
&=|G_{1,2}|\displaystyle\sum\limits_{i=2}^nx_i,
\end{align*}
since $G$ permutes $X$. Note that $y_i\neq 0$ for all $1\leq i\leq k$. Next taking a $g\in G$, which permutes $1,2$ (this exists since the action is doubly transitive) we see that $(y_1-y_2)(1,-1,0,\ldots,0)\in U$, which finishes the proof. 
\end{proof}
\begin{corollary}\label{s002}
The representation induced by the action of $\GL(3,2)=\PGL(3,2)$ on $\p^2(\f_2)$ has an irreducible degree $6$ component over $\f_{p^k}$, for $p\geq 11$.
\end{corollary}
\begin{proof}We know that the action of $\GL(3,2)$ on $\p^2(\f_2)$ is doubly transitive (see \cite[pp. 124]{H}). Since $G_{1,2}$ is a subgroup of $\GL(3,2)$ and $p\nmid |G|$, the result follows from Lemma \ref{s001}.
\end{proof}
\begin{corollary}\label{s003}
The representation induced by the action of $\GL(3,2)\cong\PGL(2,7)$ on $\p^1(\f_7)$ has an irreducible degree $7$ component over $\f_{p^k}$, for $p\geq 11$.
\end{corollary}
\begin{proof}
The action of the group $\PGL(2,7)$ on $\p^1(\f_7)$, is transitive, as well as doubly transitive (see \cite[pp. 157]{H}).  We see that $p\nmid |G_{1,2}|$, as $G_{1,2}$ is a subgroup of $\PGL(3,2)$ and $p\nmid 168$.
\end{proof}
\smallskip
\begin{remark}
Note that in Corollaries \ref{s002} and \ref{s003} the prime $p$ can be chosen lesser than $11$.
\end{remark}
\begin{remark}
Using Lemma \ref{s001}, it can be seen that the regular representation of the symmetric group $S_n$, decomposes into the trivial representation and an irreducible representation of degree $n-1$ over the field $\f_{p^k}$, whenever $p>n$.
\end{remark}
\smallskip
\begin{lemma}\label{s004}
Let $A_i$, $1\leq i\leq n$ be a family of unital algebra with unit $1_i$ and $\mathcal{D}_i$ be the set of representatives of simple $A_i$-modules. Then any simple $\bigoplus\limits_{i=1}^nA_i$-module is of the form $\bigoplus\limits_{i=1}^nM_{i}$, where not all $M_{i}$'s are zero and $M_i\in\mathcal{D}_i$.
\end{lemma}
\begin{proof}
Since $1_{\bigoplus\limits_{i=1}^nA_i}=\displaystyle\sum\limits_{i=1}^n 1_{A_i}$ and hence for any $\bigoplus\limits_{i=1}^nA_i$-module $M$, we have
\begin{align*}
    M&=M\cdot1_{\bigoplus\limits_{i=1}^nA_i}\bigoplus\limits_{i=1}^nA_i\\
    &=\bigoplus_{i=1}^nMA_i.
\end{align*}
\end{proof}
\begin{lemma}\cite[Example 3.3]{P}\label{s005}
For any division algebra (in particular field) $D$, the only simple $M(n,D)$ module is $D^n$ upto isomorphism.
\end{lemma}
\smallskip
\begin{corollary}\label{s006}
Let $G$ be a finite group, $k$ be a finite field of characteristic $p>0$, $p\nmid |G|$. Then if there exists an irreducible representations of degree $n$ over $k$, then one of the component of $kG$ is of the form $M(n,k)$.
\end{corollary}
\begin{proof}
Since $p\nmid|G|$, by Maschke's theorem $kG$ is semisimple. Hence by Artin–Wedderburn theorem we have that 
\begin{align*}
    kG=\bigoplus\limits_{i=1}^nM({n_i},k_i),
\end{align*}
where $k_i$'s are finite extensions of $k$ (hence a field). It follows from Lemma \ref{s004} and Lemma \ref{s005} that for some $i$, we have $n_i=n,k_i=k$. Hence the result follows.
\end{proof}
\begin{corollary}\label{s010}
Two of the components of the group algebra $\f_q\SL(3,2)$ are $\m(6,\f_q),\m(7,\f_q)$.
\end{corollary}
\begin{proof}
This follows immediately from Corollaries \ref{s002}, \ref{s003} and \ref{s006}.
\end{proof}
\section{Units in $\f_q\SL(3,2)$}
\begin{proposition}\label{n004}
Let $\f_q$ be a field of characteristic $p$ and $p$ $\geq$ $11$ and $q= p^k$. Let $G$ be the group $\SL(3,2)$. Then the Artin-Wedderburn decomposition of $\f_qG$ is one of the following:
\begin{center}
\centering
$    \f_q\oplus \bigoplus\limits_{i=1}^5M(n_i,\f_q)$,\\
 $   \f_q\oplus \bigoplus\limits_{i=1}^3M(n_i,\f_q)\oplus M(n_4,\f_{q^2})$ 
\end{center}
\end{proposition}
\begin{proof}
Since $p\nmid |G|$, by Maschke's theorem we have $\f_qG$ is semisimple and hence $J(\f_qG)$ is zero. By its Wedderburn decomposition we have $\f_qG$ is isomorphic to $\bigoplus\limits_{i=1}^nM({n_i},K_i)$, where $n_i > 0$ and $K_i$ is a finite extension of $\f_q$, for all $1\leq i\leq n$.

Firstly from Lemma \ref{n003}, we have \begin{equation}\label{ne001}
    \f_qG \cong \f_q \bigoplus\limits_{i=1}^{n-1}M({n_i},K_i),
\end{equation}
taking $h$ to be the augmentation map.
Now to compute these $n_i$'s and $K_i$'s we calculate the cyclotomic $\f_q$ classes of $G$. We do this in $6$ cases, for $k=6l+i$ , $0\leq i \leq 5$. Note that $p$ can have the following possibilities, being a prime
\begin{align*}
    p\in\{\pm 1\}&\mod 4,\\
    p\in\{\pm 1\}&\mod 3,\\
    p\in\{\pm 1,\pm 2,\pm 3\}&\mod 7.
\end{align*}
\begin{enumerate}
    \item \underline{The case $(k= 6l)$:} In this case $p^k \equiv 1\mod 7,p^k \equiv 1\mod4\text{ and }p^k \equiv 1\mod 3$, hence $p^k \equiv 1\mod84$ (using Chinese Remainder theorem). Thus $I_{\f_q}=\{ 1\}$ and $S_{\f_q}(\gamma_g)=\{\gamma_g\}$ for all $g \in G$. Thus by Lemma \ref{n001}, Lemma \ref{n002} and Equation \ref{ne001}
\begin{align*}
\f_qG\cong    \f_q\oplus \bigoplus\limits_{i=1}^5M(n_i,\f_q).
\end{align*} 
When such a decomposition arises, we say that $(p,k)$ is of type $1$.
 \item \underline{The case $(k= 6l+1)$:} In this case if $p\equiv \pm 1\mod 3,p\equiv\pm 1\mod 4$ and $p \equiv  1,2,-3\mod7$, $S_{\f_q}(\gamma_g)$ = $\{\gamma_g\}$ for all $g \in G$, because we have 
 \begin{align*}
     [\alpha_2]=[\alpha_2^{-1}],     [\alpha_3]=[\alpha_3^{-1}],     [\alpha_4]=[\alpha_4^{-1}].
 \end{align*}
 Once again by  Lemma \ref{n001} and Lemma \ref{n002} and Equation \ref{ne001} \begin{align*}
\f_qG\cong    \f_q\oplus \bigoplus\limits_{i=1}^5M(n_i,\f_q).
\end{align*} i.e $(p,k)$ is of type $1$.
Now if $p \equiv  -1,-2,3\mod7$, then we get that $S_{\f_q}(\gamma_g)=\{\gamma_g\}$ for $g \in \{\alpha_1,\alpha_2,\alpha_3,\alpha_4 \}$ and  $S_{\f_q}(\gamma_g) = (\gamma_g , \gamma_{g^{-1}})$ when $g\in\{\alpha_5,\alpha_6\}$ since $[\alpha_5]$ $\neq$ $[\alpha_5^{-1}]$.
  Hence in this case we have
\begin{align*}
\f_qG\cong    \f_q\oplus \bigoplus\limits_{i=1}^3M(n_i,\f_q)\oplus M(n_4,\f_{q^2}).
\end{align*} When such a decomposition arises, we say that $(p,k)$ is of type $2$.
\end{enumerate}
\smallskip
It can be further shown using Equation \ref{ne010} and Equation \ref{ne011} that $(p,k)$ is either of type $1$ or $2$. The 
possibilities are listed in the table below.
\begin{center}
\begin{tabular}{|c|c|c|}
\hline
$\text{$ p\mod 7$}$&$\text{$k$}$&$\text{Type of $(p,k)$}$\\
\hline $\pm 1,\pm 2 ,\pm 3$ & $6l$  & $1$\\
\hline $1,2,-3$&$6l+1$ &$1$\\
\hline$-1,-2,3 $&$6l+1$ &$2$\\
\hline$\pm 1,\pm 2 ,\pm 3$  &$6l+2$&$1$\\
\hline $1,2,-3$&$6l+3$&$1$\\
\hline $-1,-2,3$&$6l+3$ &$2$\\
\hline$\pm 1,\pm 2 ,\pm 3$  &$6l+4$ &$1$\\
\hline $1,2,-3$&$6l+5$ &$1$\\
\hline $-1,-2,3$&$6l+5$ &$2$\\
\hline
\end{tabular}
\end{center}
\end{proof}
\begin{proposition}\label{s007}
We have $(n_1,n_2,n_3,n_4,n_5,n_6)=(1,6,7,8,3,3)$ up to some permutation.
\end{proposition}
\begin{proof}
By Corollary \ref{s010}, we have that for some $n_i=6,n_j=7$ for some $i,j\in\{1,2,\ldots,6\}$. Let us assume $n_2=6,n_3=7$. Since $n_1=1$, we are left with the equation $n_4^2+n_5^2+n_6^2=82$, with all $n_i>0$. Since the only possibility is $8^2+3^2+3^2$, we are done.
\end{proof}
\begin{proposition}\label{s008}
Let $\f_q$ be a field of characteristic $p$ and $p$ $\geq$ $11$ and $q= p^k$. Let $G$ be the group $\SL(3,2)$. Then the Wedderburn decomposition of $\f_qG$ is as follows :
\begin{align*}
    \f_q\oplus \m(6,\f_q)\oplus\m(7,\f_q)\oplus\m(8,\f_q)\oplus
    \m(3,\f_q)^2&\text{ if $(p,k)$ is of type $1$},\\
    \f_q\oplus \m(6,\f_q)\oplus\m(7,\f_q)\oplus\m(8,\f_q)\oplus\m(3,\f_{q^2})&\text{ if $(p,k)$ is of type $2$}.
\end{align*}
\end{proposition}
\begin{proof}
Follows immediately from Proposition \ref{n004} and Proposition \ref{s007}. 
\end{proof}
\medskip
\begin{theorem}\label{s009}
Let $\f_q$ be a field of characteristic $p$ and $p$ $\geq$ $11$. Let $G$ be the group $\SL(3,2)$. Then the unit group $\mathcal{U}(\f_qG)$ is as listed in the following table:
\begin{center}
\begin{tabular}{|c|c|c|}
\hline
$\text{$ p\mod 7$}$&$\text{$k$}$&$\mathcal{U}(\f_q\SL(3,2))$\\
\hline $\pm 1,\pm 2 ,\pm 3$ & $6l$  & $\f_q^\times\oplus \GL(6,\f_q)\oplus\GL(7,\f_q)\oplus\GL(8,\f_q)\oplus
    \GL(3,\f_q)^2$\\
\hline $1,2,-3$&$6l+1$ &$\f_q^\times\oplus \GL(6,\f_q)\oplus\GL(7,\f_q)\oplus\GL(8,\f_q)\oplus
    \GL(3,\f_q)^2$\\
\hline$-1,-2,3 $&$6l+1$ &$\f_q^\times\oplus \GL(6,\f_q)\oplus\GL(7,\f_q)\oplus\GL(8,\f_q)\oplus
    \GL(3,\f_{q^2})$\\
\hline$\pm 1,\pm 2 ,\pm 3$  &$6l+2$&$\f_q^\times\oplus \GL(6,\f_q)\oplus\GL(7,\f_q)\oplus\GL(8,\f_q)\oplus
    \GL(3,\f_q)^2$\\
\hline $1,2,-3$&$6l+3$&$\f_q^\times\oplus \GL(6,\f_q)\oplus\GL(7,\f_q)\oplus\GL(8,\f_q)\oplus
    \GL(3,\f_q)^2$\\
\hline $-1,-2,3$&$6l+3$ &$\f_q^\times\oplus \GL(6,\f_q)\oplus\GL(7,\f_q)\oplus\GL(8,\f_q)\oplus
    \GL(3,\f_{q^2})$\\
\hline$\pm 1,\pm 2 ,\pm 3$  &$6l+4$ &$\f_q^\times\oplus \GL(6,\f_q)\oplus\GL(7,\f_q)\oplus\GL(8,\f_q)\oplus
    \GL(3,\f_q)^2$\\
\hline $1,2,-3$&$6l+5$ &$\f_q^\times\oplus \GL(6,\f_q)\oplus\GL(7,\f_q)\oplus\GL(8,\f_q)\oplus
    \GL(3,\f_q)^2$\\
\hline $-1,-2,3$&$6l+5$ &$\f_q^\times\oplus \GL(6,\f_q)\oplus\GL(7,\f_q)\oplus\GL(8,\f_q)\oplus
    \GL(3,\f_{q^2})$\\
\hline
\end{tabular}.
\end{center}
\end{theorem}
\begin{proof}
This follows immediately from Proposition \ref{s008} and the fact that given two rings $R_1,R_2$, we have $(R_1\times R_2)^\times=R_1^\times\times R_2^\times$.
\end{proof}
\begin{remark}
Theorem \ref{s009} holds for $p=5$ as well.
\end{remark}
\section{Concluding remarks}
We have used some techniques of character theory to reduce the number of possibilities for $n_i$'s. The book \cite{DL} deals a good portion of ordinary representation theory over finite field. From exercise at the end of \textsection $4$, we have
\begin{remark}
Let $G$ be a finite group and $k$ is a field such that $\textup{char} k\nmid|G|$. Assume $\{V_i:1\leq i\leq r\}$  to be full set of representatives of non-isomorphic irreducible $kG$-modules. Then $k$ is a splitting field of $G$ if and only if
\begin{align*}
    |G|=\displaystyle\sum\limits_{i=1}^r\dim_k(V_i)^2.
\end{align*}.
\end{remark}
Using this we conclude that
\begin{remark}
For $G=\GL(3,2)$, the field $\f_{q}$, where $q=p^k$, where either $p=5$ or $p\geq 11$ is a splitting field of $G$ if and only if $(p,k)$ is of type $1$.
\end{remark}

\end{document}